\documentclass[11pt]{amsart}

\usepackage{comment}

\newtheorem{theorem}{Theorem}[section]

\newtheorem{corollary}[theorem]{Corollary}

\newtheorem{lemma}[theorem]{Lemma}

\newtheorem{proposition}[theorem]{Proposition}

\newtheorem*{theorem*}{Theorem}
\newtheorem*{corollary*}{Corollary}

\theoremstyle{definition}

\usepackage[colorlinks, bookmarks=true]{hyperref}
\usepackage{color,graphicx,shortvrb}

\usepackage{enumerate}
\usepackage{amsmath}
\usepackage{amssymb}

\def\query#1{\setlength\marginparwidth{80pt}%
\marginpar{\raggedright\fontsize{10}{10}\selectfont\itshape
\hrule\smallskip
{\textcolor{red}{#1}}\par\smallskip\hrule}}

\newcommand{\R}{\mathbb{R}}

\newcommand{\triple}[1]{{\left\vert\kern-0.25ex\left\vert\kern-0.25ex\left\vert #1 
    \right\vert\kern-0.25ex\right\vert\kern-0.25ex\right\vert}}

\begin{document}

\numberwithin{equation}{section}


\title[Order extreme points]{Corrigendum: order extreme points \\ and solid convex hulls}



\author{A.~Ianina and T.~Oikhberg}

\address{A.I. and T.O.: Dept.~of Mathematics, University of Illinois, Urbana IL 61801, USA}
\email{aianina2@illinois.edu, oikhberg@illinois.edu}

 \author{M.A.~Tursi}

\address{M.A.T.: Independent researcher}
\email{maryangelica.tursi@gmail.com}

\date{\today}





\maketitle

\parindent=0pt
\parskip=3pt

\begin{abstract}
We correct some errors found in [T.~Oikhberg and M.A.~Tursi,
{\it Order extreme points and solid convex hulls},
The Mathematical Legacy of Victor Lomonosov (ed.~R.~Aron et.al.), de Gryuter, 2020, 297--315.]
\end{abstract}

\maketitle
\thispagestyle{empty}


The goal of this note is two-fold: (i) to correct some annoying typos in \cite{OiTu}, and (ii) to correct some of the proofs. The enumeration of sections and statements (theorems, definitions, examples etc.) in the Arxiv preprint is different from that in the published version; the numbers from the published version will be given in square brackets.

{\bf (1)} In several statements in Sections 3 [19.3] and 4 [19.4], the convexity assumptions are omitted. Specifically:

$\bullet$ Proposition 3.1 [19.7] -- Separation -- should read: \\
Suppose $\tau$ is a sufficiently rich topology on a Banach lattice $X$, and $A \subset X_+$ is a $\tau$-closed positive-solid bounded {\bf convex} subset of $X_+$. Suppose, furthermore, $x \in X_+$ does not belong to $A$. Then there exists $f \in X^\tau_+$ so that $f(x) > \sup_{a \in A} f(a)$.

$\bullet$ Theorem 4.1 [19.10] -- ``Solid'' Krein-Milman -- should read: \\
Any $\tau$-compact positive-solid {\bf convex} subset $A$ of $X_+$ coincides with the $\tau$-closed positive-solid convex hull of its order extreme points.

$\bullet$ Corollary 4.2 [19.11] should read: \\
Any $\tau$-compact solid {\bf convex} subset of $X$ coincides with the $\tau$-closed solid convex hull of its order extreme points.

Without convexity, these results may fail. For instance, for Corollary 4.2, we can take $X = (\R^2, \| \cdot \|_\infty)$, and $A = \{(x,0) : |x| \leq 1\} \cup \{(0,y) : |y| \leq 1\}$. The set $A$ is closed when $\tau$ is the norm topology and solid, but the convex hull of $A$ is much larger than $A$ itself.

{\bf (2)} There is a hole in the proof of Theorem 7.1 [19.30]: on page 17 line -2 [page 313 line -6], it is claimed that, for any $x \in C_+ \backslash \{0\}$, we have $x \wedge u \neq 0$. The current argument doesn't permit to reach this conclusion. However, this can be fixed as follows:

\begin{proof}[Patch on the proof of Theorem 19.30.]
Denote by $B \subset X$ the band generated by $u$ in $X$. By \cite[Proposition 1.2.3]{M-N}, $B$ is closed (and solid); by \cite[Proposition 1.2.3]{M-N}, $B_+$ consists of all $z \in X_+$ so that $z = \vee_n (z \wedge nu)$.
As $u$ is a quasi-interior point of $Y$, we conclude that $C' \subset Y \subset B$, hence $C \subset B$. So, $x = \vee_n (x \wedge nu)$, which implies $x \wedge nu \neq 0$ for large enough $n$. As $x \wedge nu \leq n (x \wedge u)$, we conclude that $x \wedge u \neq 0$.
%
\end{proof}

{\bf (3)} The proof of the implication $(1) \Rightarrow (2)$ in Proposition 4.3 [19.12] is incorrect: the set $D$ constructed on page 10 [306], first paragraph, is not solid, hence it is not guaranteed to have any order extreme points.

Proposition 4.3 [19.12] can be established in a rather roundabout way. More specifically, Theorem 7.1 [19.30] implies:

\begin{corollary*}
For a Banach lattice $X$, the following statements are equivalent:
\begin{enumerate}
	\item $X$ has the Radon-Nikodym Property $($RNP$)$.
	\item $X$ has the Krein-Milman Property $($KMP$)$.
	\item $X$ has the Solid Krein-Milman Property $($SKMP$)$ -- that is, any closed bounded convex solid subset of $X$ is the closure of the solid convex hull of its order extreme points.
	\item Any closed bounded convex solid subset of $X$ has an order extreme points.
\end{enumerate}
\end{corollary*}

Above, item (4) is a \emph{verbatim} repetition of Proposition 4.3[19.12](1), while item (3) is nothing but a restatement of Proposition 4.3[19.12](2).

\begin{proof}
The implication $(3) \Rightarrow (4)$ is immediate, and $(2) \Rightarrow (3)$ follows by combining Theorem 2.2 [19.2] with Corollary 2.4 [19.4]. The equivalence $(1) \Leftrightarrow (2)$ is known, see the references cited in \cite{OiTu}.

The proof of  Theorem 7.1 [19.30] actually produces, for any Banach lattice $X$ failing the RNP, a closed bounded convex solid $E \subset X$ without any order extreme points. This gives $\lnot (1) \Rightarrow \lnot (4)$, or equivalently, $(4) \Rightarrow (1)$.
\end{proof}

\end{document}